\documentclass[12pt]{amsart}
\usepackage[T1]{fontenc}
\usepackage{amsmath}
\usepackage{amssymb, amsthm}
\usepackage[foot]{amsaddr}
\usepackage{amsfonts}
\usepackage{newtxtext, newtxmath}
\usepackage{enumitem}
\usepackage[colorlinks=true]{hyperref}
\usepackage{tikz}
\usepackage{pgfplots}
\usepackage[normalem]{ulem}
\usepackage{xcolor}
\usepackage{comment}

\numberwithin{equation}{section}

\pgfplotsset{compat=1.18}

\newtheorem{theorem}{Theorem}[section]
\newtheorem*{theorem*}{Theorem}
\newtheorem{lemma}[theorem]{Lemma}
\newtheorem{proposition}[theorem]{Proposition}
\newtheorem*{proposition*}{Proposition}

\theoremstyle{definition}

\newtheorem{definition}[theorem]{Definition}
\newtheorem{remark}[theorem]{Remark}

\title{Continuous Persistence Landscapes}
\author{Peter Bubenik$^\dagger$ and Wanchen Zhao$^\dagger$}
\address{$^\dagger$ Department of Mathematics, University of Florida}

\newcommand{\R}{\mathbb{R}}
\newcommand{\B}{\mathcal{B}}

\newcommand{\E}{\mathcal{E}}
\newcommand{\Q}{\mathbb{Q}}
\newcommand{\N}{\mathbb{N}}
\newcommand{\Z}{\mathbb{Z}}

\newcommand{\kf}{\mathbb k}
\newcommand{\Rleq}{\R_\leq^2}
\newcommand{\Rl}{\R_{<}^2}

\newcommand{\veck}{\mathsf{vect}_{\mathbb k}}

\newcommand{\Qth}{Q_{t,h}}
\newcommand{\cQth}{\overline{Q}_{t,h}}
\newcommand{\rk}{\operatorname{rank}}

\newcommand{\Rpt}{(\Rleq, \Delta)}
\newcommand{\Wrk}{W_1^{\rk}}
\newcommand{\drank}{d_{\rk}}

\begin{document}

\begin{abstract}
    As the size of data increase, persistence diagrams often exhibit structured asymptotic behavior, converging weakly to a Radon measure. However, conventional vector summaries such as persistence landscapes are not well-behaved in this setting, particularly for diagrams with high point multiplicities. We introduce continuous persistence landscapes, a new vectorization defined on a special class of Borel measures, which we call q-tame measures. It includes both the persistence diagrams and their weak limits. Our construction generalizes persistence landscapes to a measure-theoretic setting, preserving the intrinsic structure of persistence measures. We show that this vector summary is bijective and $L^1-$stable under mild assumptions, and that the original measure can be uniquely reconstructed. This approach gives a more faithful description of the shape of data in the limit and provides a stable, invertible way to analyze topological features in large systems. 
\end{abstract}

\maketitle

\section{Introduction} \label{intro}


Understanding the asymptotic behavior of persistence diagrams (PD) and their vector summaries is of both theoretical and practical interest. As data size and complexity increase, PDs tend to exhibit structured or clustered patterns \cite{Hiraoka2016}. Some works have analyzed the asymptotic behavior of persistence landscapes~\cite{pl} and related summaries \cite{chazal15}, but these approaches face limitations: when the points have high multiplicities, the first many layers of the persistence landscapes may be identical, capturing only the outer envelope of the support of the rank function. Moreover, embeddings of PDs into Hilbert spaces fail to preserve their intrinsic metric structure \cite{BW2020}, making it difficult to interpret or recover asymptotic limits in the original topological sense.

Convergence results for PDs have been established under various metrics and sampling methods, showing that PDs converge weakly to Radon measures supported on $\{(x,y)\in \R^2 \ : \ x<y\}$ \cite{Divol_Chazal_2020, Divol-Lacombe2021}. However, a vectorization that is meaningful in this broader measure-theoretic setting has been lacking.

We introduce the continuous persistence landscape, a vector summary defined for a class of Borel measures on $\{(x,y)\in \R^2\ : \ x<y\}$ that we call $q$-tame measures, which includes both PDs and their weak limits. We prove that this vectorization is bijective and $L^1$-stable under mild assumptions. Our construction not only extends persistence landscapes to a measure-theoretic setting but also allows one to reconstruct the original measure via Carathéodory's Extension Theorem. Thus, continuous persistence landscapes provide a robust and invertible representation of persistence measures, bridging the gap between discrete PDs and their limits.

By considering this more general version of persistence landscapes, we discover a previously overlooked property of persistence landscapes~\cite{Bubenik:2020b}. Together with the new property, we get a full characterization of continuous persistence landscapes, stated in Proposition \ref{pl properties}. These are the necessary and sufficient conditions for reconstructing the $q-$tame measures.
When taking $N$ samples of PDs of finitely many points, the average persistence landscapes satisfy a law of large numbers and Central Limit theorem as $N$ approaches $\infty$ \cite{pl}. We compare the continuous persistence landscapes of the mean measure of these samples to the average persistence landscapes, and show that, in general, they are incomparable. In Proposition \ref{apl and cpl}, we give a necessary condition for the average persistence landscapes to be an upper bound of the continuous persistence landscapes.

\subsection{Related work and our contribution}
Divol and Chazal showed the PDs converge to a Radon measure with density supported on $\Rl$ with respect to the $2$-norm on the plane \cite{Divol_Chazal_2020}. Divol and Lacombe proved convergence and provided the convergence rate of PDs with respect to the $p-$Wasserstein distance with the $p-$norm on the plane, which is widely used in TDA \cite{Divol-Lacombe2021}. Chazal et al. showed the average persistence landscapes are stable with respect to changes of the underlying probability measure on point clouds \cite{chazal15}. This approach differs from ours in that the authors randomly subsample point clouds of fixed size and compute the average persistence landscapes, which estimate the expected persistence landscape. Our approach does not fix the size of a point cloud and allows the PDs to have infinitely many points. We show that if the PDs have finite 1-Wasserstein distance to the empty diagram, then their continuous persistence landscapes are $L^1-$stable with respect to the 1-Wasserstein distance.

\section{Background} \label{bg}

We introduce notation and review basic notions in persistent homology and measure theory. 
We will use the following subsets of $\R^2$.
Let $\Rl = \{ (x,y) \ : \ x < y\}$,
$\Delta = \{ (x,y) \ : \ x=y\}$, and 
$\Rleq = \Rl \cup \Delta$.
As needed, we equip these with their subspace topology.
For $t \in \R$ and $h \geq 0$, 
let $\Qth = (-\infty,t-h) \times (t+h,\infty)$ and
$\cQth = (-\infty,t-h] \times [t+h,\infty)$.

\subsection{Measure Theory}

The Borel $\sigma-$algebra on $\Rl$ is the smallest $\sigma-$algebra generated by the open sets. 
One set of generators of this Borel $\sigma-$algebra is given by
$\{\Qth \ : \ t \in \R, h \geq 0 \}$.
A Borel measure is a measure defined on this Borel $\sigma-$ algebra. 

Let $X$ be a set and $\Sigma$ be a $\sigma-$algebra over $X$. A set function $\mu:\Sigma \to [0, \infty]$ is a \emph{measure} if \textbf{(1)} $\mu(\emptyset)=0$, and
\textbf{(2)} ($\sigma-$additive) for every countable set $\{ E_i \}_{i=1}^{\infty}$ of pairwise disjoint sets, \[ \mu (\overset{\infty}{\underset{i=1}{\bigcup}} E_i) = \overset{\infty}{\underset{i=1}{\sum}} \mu (E_i)\]

Let $\Sigma_0$ be an algebra on $X$. 
A function $\mu_0: \Sigma_0 \to [0,\infty]$ is a \textit{premeasure} if $\mu_0$ satisfies \textbf{(1)} and \textbf{(2)} above. A function $\mu':2^{X}\to [0,\infty]$ is an \textit{outermeasure} if $\mu'(\emptyset) = 0$, $\mu'(A)\leq \mu'(B)$ whenever $A\subset B$, and $\sigma-$subadditive, i.e. for every countable set $\{ E_i \}_{i=1}^{\infty}$ of pairwise disjoint sets, \[ \mu (\overset{\infty}{\underset{i=1}{\bigcup}} E_i) \leq \overset{\infty}{\underset{i=1}{\sum}} \mu (E_i). \] A premeasure naturally give rise to an outermeasure by defining 
\[ \mu'(S) = \inf \Bigl\{ \overset{\infty}{\underset{i=1}{\sum}} \mu_0(A_i): A_i\in \Sigma_0, S \subset \overset{\infty}{\underset{i=1}{\bigcup}} A_i \Bigr\}.
\]
Given an outermeasure $\mu'$, a subset $E\subset X$ is \textit{$\mu'-$measurable}, or \textit{outermeasurable} when the outermeasure is clear in context, if for every $F\subset X$, $\mu' (F) = \mu'(F\cap E) + \mu'(F \setminus E)$.

\begin{theorem}
    (Carathéodory's Extension Theorem) Let $\Sigma_0, \Sigma$ be the algebra and $\sigma-$algebra over a set $X$. A premeasure $\mu_0:\Sigma_0\to[0,\infty]$ can be extended to a measure $\mu:\Sigma \to [0,\infty]$. Additionally, if $X$ is $\sigma-$finite, the extension is unique.
\end{theorem}

\begin{definition} \label{wp def for measures}
    Let $\mu, \nu$ be measures on $\Rl$. Let $d$ be a metric on $\Rleq$. A coupling $\gamma$ between $\mu, \nu$ is a Borel measure on $\Rleq \times \Rleq$ such that the push forward measures under the projection maps, $p_1, p_2: \Rleq \times \Rleq \to \Rleq$, recover $\mu, \nu$ on $\Rl$ respectively. The Wasserstein distance between $\mu, \nu$ is given by: \[ W_{p}^{d}(\mu, \nu) = \underset{\gamma}{\inf} \int d(x,y)^p d\gamma ^{1/p} \]
\end{definition}


\subsection{Persistent Homology} \label{ph}

We consider $\R$ with its usual linear order to be a category $\textbf{R}$,
with objects $\{ r\in \R \}$ and morphisms corresponding to $ r\leq s$ in $\R$. Given a field $\kf$, let $\veck$ denote the category with objects being finite-dimensional $\kf-$vector spaces and morphisms being $\kf-$linear maps. A persistence module $M$ is a functor from $\R_{\leq}$ to $\veck$, which consists of vector spaces $\{ M_r \ : \ r \in \R \}$ and $\kf-$linear maps $\{ M_{r\leq s}\ : \ r\leq s \in \R_{\leq}\}$. We drop the field coefficient if it is clear from the context.

Let $I$ be an interval in $\R$. Its corresponding interval module $M$ consists of objects $ M_r = \kf$ if $r\in I $ and $M_t = 0$ if $t\notin I$, and morphisms $M_{r\leq t} = Id$ if $r\le t\in I$ and $M_{r\leq t} = 0$ if not. 
Pointwise-finite-dimensional persistence modules can be uniquely decomposed into a direct sum of interval modules~\cite{crawley}. 

The decomposition can be represented via a \textit{persistence diagram (PD)} by assigning each interval module with interval endpoints $a,b$ to a point $(a,b)\in \Rleq$. A persistence diagram can be viewed as a discrete measure, which is a Borel measure, on $\Rl$. In this paper, we only consider bounded interval modules, so persistence diagrams do not have points at infinity.

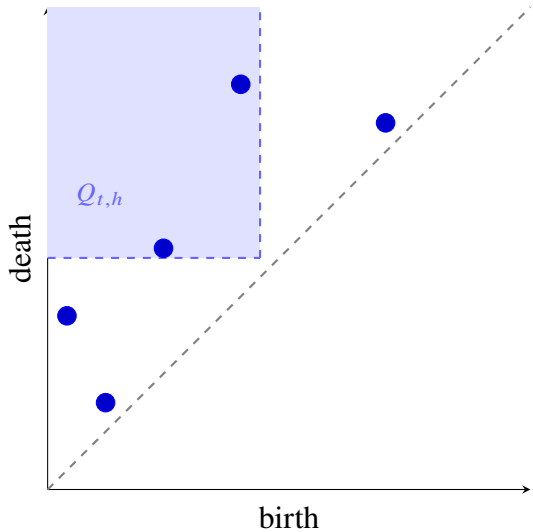
\begin{figure}[ht]
  \centering
  \begin{tikzpicture}
    \def\t{2.3}  
    \def\h{0.1}  
    \pgfmathsetmacro{\xb}{\t-\h}
    \pgfmathsetmacro{\yb}{\t+\h}

    \def\Xmax{5}
    \def\Ymax{5}

    \begin{axis}[
      width=8cm, height=8cm,
      xmin=0, xmax=\Xmax, ymin=0, ymax=\Ymax,
      axis lines=left,
      xlabel={birth}, ylabel={death},
      xtick=\empty, ytick=\empty,
      grid=both, grid style={dashed,gray!20},
      axis equal image,
      clip=false
    ]

    \addplot[domain=0:5, samples=2, gray, thick, dashed] {x};

    \addplot [
      draw=none,
      fill=blue!12
    ] coordinates {
      (0,\yb)
      (\xb,\yb)
      (\xb,\Ymax)
      (0,\Ymax)
    };

    \addplot[blue!60, thick, dashed] coordinates {(\xb,\yb) (\xb,\Ymax)};

    \addplot[blue!60, thick, dashed] coordinates {(0,\yb) (\xb,\yb)};

    \node[blue!60, font=\footnotesize]
      at (axis cs:{0.25*\xb},{\yb+0.25*(\Ymax-\yb)}) {$\Qth$};

    \addplot+[only marks, mark=*, mark size=3.5pt] coordinates {
      (0.2,1.8)
      (0.6,0.9)
      (1.2,2.5)
      (2.0,4.2)
      (3.5,3.8)
    };

    \end{axis}
  \end{tikzpicture}

  \caption{A persistence diagram with $\Qth = (-\infty, t-h)\times(t+h, \infty)$ shaded. 
  }
  \label{fig:pd-Qth-shaded}
\end{figure}

A persistence module is $q-$\textit{tame} if every non-identity linear maps has finite rank, i.e. $\forall r<s, \rk M_{r\leq s} < \infty$. 

\section{Persistence Measures and continuous persistence landscapes} \label{pm and cpl section}


In this section we define continuous persistence landscapes for $q$-tame measures, which includes Radon measures on $\Rl$.
%

\begin{definition}
    Let $\mu$ be a Borel measure on $\Rl$. We say $\mu$ is \emph{$q-$tame} if $ \forall t\in \R, h \geq 0$, $\mu(\Qth) < \infty$.
\end{definition}

We define persistence measures to be $q-$tame measures on $\Rl$. Observe $q-$tame measures are Radon measures on $\Rl$, because they assign finite weight to the rectangles $\{(a,b)\times (c,d) \subset \Rl: b\le c\}$. Every compact set is covered by finitely many such rectangles, so it has finite weight. The measure is inner regular on open sets because every open set can be written as an increasing sequence of compact sets in $\Rl$ with finite weight. 


\subsection{Continuous persistence landscapes}

We define a map from the set of persistence measures to a Banach space. We prove that this map is bijective and has a nice description for its image. 

The persistence landscapes are function on $\Z_+ \times \R$ where the first parameter refers to the layer. We extend the definition for persistence measures as follows. 

Let $\B(\Rl)$ denote the Borel $\sigma-$algebra on $\Rl$. 

\begin{definition} \label{cpl def}
    Given a $q-$tame measure $\mu:\B(\Rl)\rightarrow \R_{\geq 0}$, its continuous persistence landscape is given by $\lambda: (0,\infty) \times \R\rightarrow [0,\infty)$ where \[\lambda(a,t)=\sup\{h > 0\ : \ \mu(\Qth) \geq a\},\] where $\sup \emptyset$ is defined to be $0$.
\end{definition}
The function $\lambda$ is well-defined because $\mu(\Qth)$ is finite and thus $\mu(\Qth)\rightarrow 0$ as $h\rightarrow \infty$.

Let $\Lambda$ denote the map from the set of $q-$tame measures to the continuous persistence landscapes. The image set will be fully described in section \ref{pl properties section}.

\subsection{Properties of Persistent Landscape:} \label{pl properties section} 

We first introduce some notation that we will use to characterize continuous persistence landscapes.

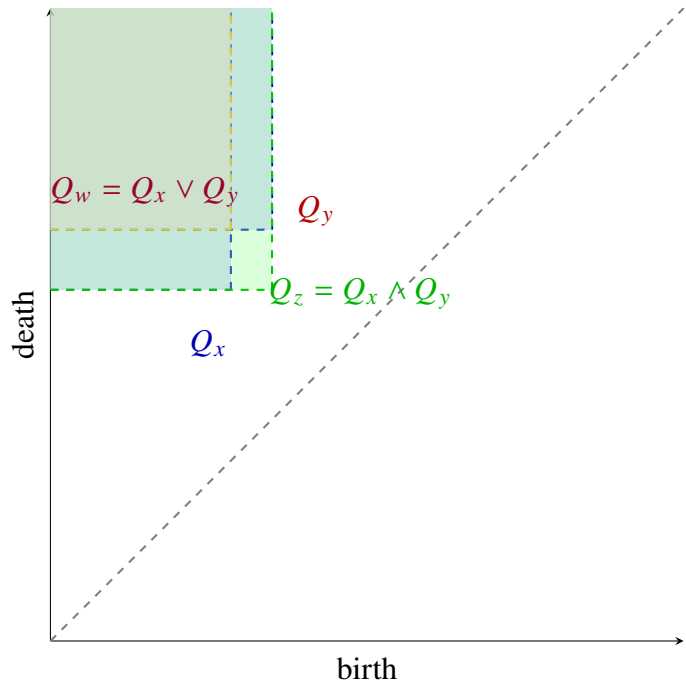
\begin{figure}[ht]
\centering
\begin{tikzpicture}

\def\tone{4.2}
\def\hone{1.35}

\def\ttwo{5.0}
\def\htwo{1.50}

\pgfmathsetmacro{\xone}{\tone-\hone}
\pgfmathsetmacro{\xtwo}{\tone+\hone}
\pgfmathsetmacro{\yone}{\ttwo-\htwo}
\pgfmathsetmacro{\ytwo}{\ttwo+\htwo}

\pgfmathsetmacro{\zone}{max(\xone,\yone)}
\pgfmathsetmacro{\ztwo}{min(\xtwo,\ytwo)}

\pgfmathsetmacro{\tthree}{(\zone+\ztwo)/2}
\pgfmathsetmacro{\hthree}{(\ztwo-\zone)/2}

\pgfmathsetmacro{\wone}{min(\xone,\yone)}
\pgfmathsetmacro{\wtwo}{max(\xtwo,\ytwo)}

\pgfmathsetmacro{\tfour}{(\wone+\wtwo)/2}
\pgfmathsetmacro{\hfour}{(\wtwo-\wone)/2}

\def\Xmax{10}
\def\Ymax{10}

\begin{axis}[
  width=10cm, height=10cm,
  xmin=0, xmax=\Xmax, ymin=0, ymax=\Ymax,
  axis lines=left,
  xlabel={birth}, ylabel={death},
  xtick=\empty, ytick=\empty,
  grid=both, grid style={dashed,gray!20},
  axis equal image,
  clip=false
]

\addplot[domain=0:\Xmax, samples=2, gray, thick, dashed] {x};

\newcommand{\ShadeQ}[3]{
  \pgfmathsetmacro{\tt}{#1}
  \pgfmathsetmacro{\hh}{#2}
  \pgfmathsetmacro{\xb}{\tt-\hh}
  \pgfmathsetmacro{\yb}{\tt+\hh}
  \addplot[draw=none, fill=#3!40, opacity=0.35] coordinates {
    (0,\yb)
    (\xb,\yb)
    (\xb,\Ymax)
    (0,\Ymax)
  };
  \addplot[#3!80!black, thick, dashed] coordinates {(\xb,\yb) (\xb,\Ymax)};
  \addplot[#3!80!black, thick, dashed] coordinates {(0,\yb) (\xb,\yb)};
}

\ShadeQ{\tone}{\hone}{blue}
\ShadeQ{\ttwo}{\htwo}{blue}
\ShadeQ{\tthree}{\hthree}{green}
\ShadeQ{\tfour}{\hfour}{yellow}

\node[blue!70!black]   at (axis cs:2.5,4.7) {$Q_x$};
\node[red!70!black]    at (axis cs:4.2,6.8) {$Q_y$};
\node[green!70!black]  at (axis cs:4.9,5.5) {$Q_z = Q_x \wedge Q_y$};
\node[purple!80!black] at (axis cs:1.5,7.1) {$Q_w = Q_x \vee Q_y$};

\end{axis}
\end{tikzpicture}

\caption{We order the quadrants with reverse containment. This choice is consistent with interval modules with the containment order.}
\label{fig: quadrants relation}
\end{figure}

Fix $x,y \in \R ^2$ where $x=(x_1, x_2), y=(y_1, y_2)$ and $x_1 = t_1-h_1, x_2 = t_1+h_1, y_1 = t_2-h_2, y_2 = t_2+h_2$. See figure \ref{fig: quadrants relation} for $x,y$, the corresponding quadrants $Q_x, Q_y$, and their meet and join under reverse containment of quadrants.


Fix $z=(z_1, z_2)=(t_3-h_3, t_3+h_3)$ where $z_1=$max$\{x_1, y_1\}$ and $z_2=$min$\{x_2, y_2\}$. Geometrically, $Q_z$ is the minimum quadrant containing $Q_x$ and $Q_y$. Similarly fix $w=(w_1, w_2)=(t_4-h_4, t_4+h_4)$ where $w_1=$min$\{x_1, y_1\}$ and $w_2=$max$\{x_2, y_2\}$. $Q_w$ is the intersection of $Q_x$ and $Q_y$, or the maximum quadrant contained by both $Q_x$ and $Q_y$.

W.l.o.g. suppose $x_1\leq y_1$. If $y_2\leq x_2$, then $Q_x\subset Q_y$. If $y_2\geq x_2$, then $z=(y_1, x_2)$, and we require the function $\lambda$ to satisfy the following property:

Given arbitrary $a, b, c >0$ such that $h_1\leq \lambda(a,t_1)$, $h_2\leq \lambda(b,t_2)$, and $h_4\leq \lambda(c,t_4)$, there exists $d>0$ such that 
\begin{align} \label{condition 3}
\text{$h_3\leq \lambda(d,t_3)$ and $d\geq a+b-c$.}
\end{align}

This property on the continuous persistence landscapes ensure the non-negativity of the premeasure we construct from the vector summary. 

\begin{proposition} \label{pl properties}
$\lambda$ has the following properties:
\begin{enumerate}
    \item $\lambda(-,t)$ is decreasing and $\lambda(a,t)\rightarrow 0$ as $a\rightarrow \infty$. 
    \item $\lambda(a,-)$ is 1-Lipschitz.
    \item $\lambda(-,t)$ is left continuous.
    \item \ref{condition 3} holds. 
\end{enumerate} 
\end{proposition}

Unlike properties (1)-(3), property (4) involves both parameters $a$ and $t$.

\begin{proof}

    (1): $\forall \epsilon >0, \{h\geq 0\ : \ \mu(\Qth) \geq a+\epsilon\} \subset \{h\geq 0\ : \ \mu(\Qth) \geq a\}$, so $\lambda(a+\epsilon,t)\leq \lambda(a,t)$.

    (2):Fix $a>0$, $t,s\in \R$. w.l.o.g., suppose $\lambda(a,t)\geq\lambda(a,s)$. 

    \textit{Case 1}: If $\lambda(a,t)\leq |t-s|$, then $\lambda(a,t)-\lambda(a,s)\leq |t-s|$.

    \textit{Case 2}: Suppose $\lambda(a,t)>|t-s|$. Fix $h>0$ such that $|t-s|+h<\lambda(a,t)$, i.e. $\forall \epsilon>0 \exists h>0$ such that $\mu(\Qth)\leq \lambda(a,t)-\epsilon$. w.l.o.g., suppose $t>s$, then $t-\lambda(a,t)<s-h$ and $s+h<t+\lambda(a,t)$. Note $\mu(Q_{s,h-|t-s|})\geq a$. Then $\lambda(a,t)-\lambda(a,s)\leq |t-s|$.

    (3): the proof directly follows from definition of the continuous landscapes and that $\Qth$ are open quadrants.

    (4): Let $\mu$ denote the measure whose landscapes is $\lambda$. Note $\mu([x_1, z_1)\times (z_2, y_2] )\ge 0$ by the additivity of measures. Since \[
    \mu([x_1, z_1)\times (z_2, y_2] ) = \mu(Q_z) -\mu(Q_x) -\mu(Q_y) + \mu(Q_w),
    \]
    then $\mu(Q_z) \ge \mu(Q_x) + \mu(Q_y) - \mu(Q_w)$. Thus $\exists d$ such that $d\geq a+b-c$, since $\lambda(k,t)=\sup\{h > 0\ : \ \mu(\Qth) \geq k\}$ for all $k>0, t\in \R$.
\end{proof}

\section{Invertibility of Continuous Persistence Landscapes}


In this section we show that the mapping given by continuous persistence landscapes is injective, and furthermore that a $q$-tame measure can be reconstructed from its continuous persistence landscape.

\begin{theorem}\label{invertible}
    Let $S$ denote the set of functions satisfying proposition \ref{pl properties}. The vectorization via continuous persistence landscapes, denoted by $\Lambda$, is bijective onto $S$.
\end{theorem}

We first show the injectivity of $\Lambda$, as well as its inverse. We begin by constructing a set function on an algebra of subsets in $\Rl$ and verifying that it is a well-defined premeasure.

\subsection{Construction of the premeasure}

Let $\E=\{\Qth: t\in\R, h\in\R_{\ge 0}\}$, and $\langle \E \rangle$ is the algebra generated by the quadrants, which is the algebra of subsets in $\Rl$. Another generating set is \[ \{ [a,b) \times(c,d]\ : \ a,b,c,d\in \R, b\le c \}\] each of which can be written as $Q_z \setminus (Q_x \cup Q_y)$, as in figure \ref{fig: quadrants relation}. 
Define
\[ 
\begin{split}
    \nu_0: \langle \E \rangle & \rightarrow [0,\infty]\\
    \Qth &\mapsto \inf \{ a\ge 0\ : \ \lambda(a,t)\le h \},
\end{split}
\] 
Since intersections of quadrants are quadrants, we define $\nu_0$ for finite unions of quadrants via the Principle of Inclusion and Exclusion:\[ \nu_0 (Q_x \cup Q_y ) = \nu_0(Q_x) + \nu_0(Q_y) - \nu_0(Q_x\cap Q_y). \] Thus, \[ \nu_0([a,b) \times(c,d]) = \nu_0 (Q_z) - \nu_0(Q_x) - \nu_0(Q_y) + \nu_0(Q_x\cap Q_y). \] By Proposition \ref{pl properties} (4), this value is nonnegative. We define $\nu_0$ for unions and relative complements of rectangles via the Principle of Inclusion and Exclusion again.  

\subsection{Compatibility with the algebra} To show $\nu_0$ is a premeasure, we verify that it respects finite unions and relative complements of rectangles. 

\begin{lemma}\label{compatibility}
    $\nu_0$ is well-defined on $\langle \E \rangle$.
\end{lemma}

\begin{proof}
    In the proof, we refer rectangles to the ones of the form $[a,b) \times(c,d]$.
    
    It suffices to show that unions and relative complements of rectangles are well-defined. 

    Observe that relative complements of rectangles can be expressed as finite unions of disjoint rectangles if non-empty. By the construction of $\nu_0$ via the inclusion-exclusion formula, $\nu_0$ is well-defined for relative complements. 

    Similarly, unions of rectangles are rectangles and have well-defined $\nu_0$ value. 
\end{proof}

\subsection{Extension to a measure}

We extend the premeasure to a measure via Carathéodory's extension theorem, and the extended measure recovers the $q-$tame measure. 

First we state a technical lemma.
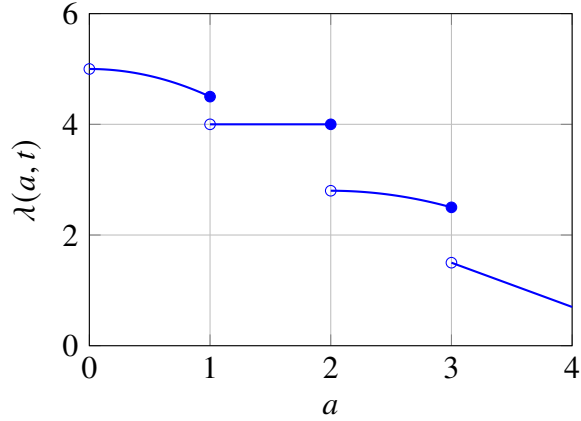
\begin{figure}[h]
\centering
\begin{tikzpicture}
\begin{axis}[
    width=8cm,
    height=6cm,
    xlabel={$a$},
    ylabel={$\lambda(a,t)$},
    xmin=0, xmax=4,
    ymin=0, ymax=6,
    grid=both,
]

\addplot[blue, thick, domain=0:1] {5 - 0.5*x^2};

\addplot[only marks, mark=o, blue] coordinates {(0,5)};
\addplot[only marks, mark=*, blue] coordinates {(1,5 - 0.5)};
\addplot[only marks, mark=o, blue] coordinates {(1,4)};

\addplot[blue, thick, domain=1:2] {4};

\addplot[only marks, mark=*, blue] coordinates {(2,4)};
\addplot[only marks, mark=o, blue] coordinates {(2,2.8)};

\addplot[blue, thick, domain=2:3] {2.8 - 0.3*(x-2)^2};

\addplot[only marks, mark=*, blue] coordinates {(3,2.5)};
\addplot[only marks, mark=o, blue] coordinates {(3,1.5)};

\addplot[blue, thick, domain=3:4] {1.5 - 0.8*(x-3)};

\end{axis}
\end{tikzpicture}
\caption{Fix $t\in \R$. We plot an example of the continuous persistence landscape $\lambda(a,t)$ varying $a$ and show it is a left-continuous generalized inverse of the function $\mu(\Qth)$, where $\mu$ is the $q-$tame measure whose continuous landscape is $\lambda$. For discrete landscapes, $\lambda(a,t)$ is a decreasing left-continuous step function, constant on $(n,n+1]$ for $n\in \N$.}
\end{figure}

\begin{figure}[h]
\centering
\begin{tikzpicture}
\begin{axis}[
    width=8cm,
    height=6cm,
    xlabel={$h$},
    ylabel={$\nu_0(\Qth)$},
    xmin=0, xmax=6,
    ymin=0, ymax=4,
    grid=both       
]


\addplot[red, thick, domain=0:1] {5 - 0.5*x^2};


\addplot[red, thick, domain=0.7:1.5] {3 + (1.5-x)/0.8};

\addplot[red, thick, domain=1.5:2.5] {3};


\addplot[red, thick, domain=2.5:2.8] {(2.8-x)/0.3)^0.5+2 };

\addplot[only marks, mark=o, red] coordinates {(4,2)};
\addplot[only marks, mark=*, red] coordinates {(4,1)};

\addplot[red, thick, domain=2.8:4] {2};

\addplot[red, thick, domain=4:4.5] {1};

\addplot[red, thick, domain=4.5:5] {(2*(5-x))^0.5};


\end{axis}
\end{tikzpicture}
\caption{Fix $t\in \R$. We plot $\nu_0(\Qth)$, the right-continuous generalized inverse of $\lambda(a,t)$, which is a decreasing function with respect to $h$. We will show $\nu_0 = \mu$ on $\langle \E \rangle$. For discrete landscapes, $\nu_0(\Qth)$ is a decreasing right-continuous step function. }
\end{figure}
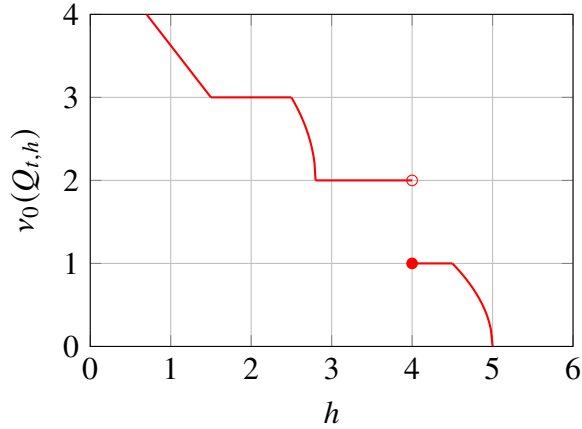

\begin{lemma} \label{useful lemma}
    Let $f$ be non-negative, decreasing, and right continuous. For each x, define $P_x(y)$ holds if \[
    \sup \{ x'\ : \ f(x') \ge y \} \leq x.
    \]
    Then $f(x) = \inf \{y\ : \ P_x(y)\}$.
\end{lemma}

\begin{proof}
    First $P_x(f(x))$ holds by construction. To show $f(x)$ is the greatest lower bound, fix $y< f(x)$. 
    
    $\sup\{x'\ : \ f(x')\geq y\} \geq x$ because $f(x)>y$. By right continuity, there exists $\epsilon>0$ such that $y\leq f(x+\epsilon)\leq f(x)$. Then sup$\{x'\ : \ f(x')\geq y\} \geq x+e > x$, i.e. $y$ doesn't satisfy $P_x$. 
\end{proof}

Now we are ready to prove the main theorem \ref{invertible}. We will show the $\sigma-$additivity of $\nu_0$ and invertibility of the continuous landscapes at the same time, by showing $\nu_0$ coincides with $\mu$ on $\langle \E \rangle$, the algebra of subsets in $\Rl$. Then by Carathéodory's extension theorem, the premeasure $\nu_0$ extends to a measure on the $\sigma-$algebra of $\Rl$, uniquely up to sets of measure 0.

\begin{proof}
    Let $t\in \R$. Since $\mu(\Qth)$ is non-negative, decreasing, and right continuous with respect to $h$, we get: \[
    \begin{split}
        \nu_0(\Qth) &= \inf \{a\geq 0\ : \ \lambda(a,t) \le h\} \\
                    &= \inf \{a\geq 0\ : \ \sup\{h > 0\ : \ \mu(\Qth) \geq a\} \leq h\}\\
                    &\overset{\ref{useful lemma}}{=} \mu(\Qth)
    \end{split}
    \]
    Then $\nu_0$ is a premeasure because it is $\sigma-$additive and assigns 0 to the empty set. By Carathéodory's extension theorem, we can extend $\nu_0$ from the algebra generated by $\E$ to the $\sigma-$algebra generated by $\E$. Because $\Rl$ is $\sigma-$finite, the extended measure is unique. Since $\nu_0$ and $\mu$ agree on the generating set, the extension recovers $\mu$.
\end{proof}

\begin{remark}
    One could consider expanding the construction for signed $q-$tame measures. The construction relies on the landscapes' properties in proposition \ref{pl properties}, and it works if one is given the Jordan decomposition of a signed $q-$tame measure. 
\end{remark}

\section{Stability} \label{stability section}

We show vectorization using continuous landscapes is $L^1-$stable with respect to the 1-Wasserstein distance on persistence measures together with a specific choice of the underlying metric introduced in \cite{BZ}.

Recall the definition of $1-$Wasserstein distance on persistence measures: \[
 W_{1}^{d}(\mu, \nu) = \underset{\gamma}{\inf} \int d(x,y) d\gamma
\]
where $d$ is a metric on $\Rleq$. Let $\Rpt$ denote the pointed topological space $(\Rleq/\Delta, \Delta)$ equipped with the quotient topology. 
Let $q: \Rleq \to \Rleq/\Delta$ be the quotient map. Note that the pullback of a (pseudo-)metric $d$ on $\Rleq / \Delta$ under $q$ is a pseudo-metric. The metric $d$ can be viewed as a cost function for transporting mass from the point $x$ to $y$. There is no canonical choice of $d$. We choose the following metric to prove $W_1-$stability of our vector summaries for $q-$tame measures.

Let $x=(x_1, x_2), y=(y_1, y_2) \in \Rleq$. The metric is given by 
\[
\drank(x,y) = \frac{1}{2} (x_2 - x_1)^2 + \frac{1}{2} (y_2 - y_1)^2 - \max(\min(x_2,y_2) - \max(x_1,y_1),0)^2 .
\]
Geometrically, $\drank(x,y)$ is the area of the symmetric difference between the 2 right triangles $x,y$ makes with the diagonal. The expression includes the minimum and maximum to account for different configurations of $x,y$. 

Intuitively, $\drank$ considers the persistence of diagram points instead of just their relative position on the plane. More persistent points are associated with higher costs to move around. Let $\Wrk$ denote the 1-Waaserstein distance together with the metric $\drank$.

\begin{theorem} \label{W1rank stability}
    \textbf{(\cite{BZ} Theorem 5.9)} Let $\mu,\nu$ be countably supported persistence diagrams such that $\Wrk(\mu,0)$, $\Wrk(\nu,0)$ are finite. Then \[
    \Vert \Lambda \mu - \Lambda \nu \Vert_1 \le \frac{1}{2}\Wrk(\mu,\nu)
    \]
\end{theorem}

To extend this stability result, we consider $q-$tame measures with countable support and finite $\Wrk$ distance to the empty diagram. These measures can each be approximated by finitely supported $q-$tame measures, $\mu',\nu'$. By denseness of $\Q$ in $\R$, we can choose $\mu',\nu'$ to have rational weights on all points in the support: \[
\mu' = \sum_{i\le m} a_i x_i \qquad \nu' = \sum_{i\le n} b_i y_i \] such that \[ \Wrk(\mu,\mu') < \epsilon \Wrk(\nu,\nu') < \epsilon
\] for any $\epsilon >0$. After multiplying by the common denominator of elements in $\{a_i\}_{i\le m}\cup \{b_i\}_{i\le n}$, the case reduces to comparing 2 finite persistence diagrams. 


\begin{theorem}
    Let $\mu,\nu$ be countably supported  $q-$tame measure such that $\Wrk(\mu,0), \Wrk(\nu,0)$ are finite. Then \[
    \Vert \Lambda \mu - \Lambda \nu \Vert_1 \le \frac{1}{2} \Wrk(\mu,\nu)
    \]
\end{theorem}

The proof follows from the proof for Theorem 5.9 in \cite{BZ} with the following observation. Let $c$ be the common denominator for the weights of points in the support of $\mu',\nu'$, \[
\Wrk(c\mu',c\nu') = c \Wrk(\mu',\nu')
\] The multiplication by $c$ scales the landscapes by $\frac{1}{c}$ for the parameter $a$. It follows from the continuous persistence landscapes definition \ref{cpl def} that \[
\Lambda(c\mu')(a,t) = \Lambda(\mu')(\frac{a}{c},t)
\].

\section{Comparison with average persistence landscapes}


Embeddings into Hilbert spaces do not preserve the metric space structure of PDs \cite{BW2020}. One natural approach to vectorize the estimation of a persistence measure is to draw PD samples from the probability measure on $\Rl$, vectorize each PD and take the mean. Chazal, Divol, Lacombe showed the empirical counterpart of a persistence measure is an optimal estimator from a minimax viewpoint \cite{Divol_Chazal_2020, Divol-Lacombe2021}. In this section, we investigate the relationship between the continuous persistence landscape of the empirical counterpart and the average persistence landscape of the samples, which satisfies the law of large numbers and central limit theorem.

Consider PDs $\{\alpha_i\}_{i=1}^{N}$ and persistence landscapes $\{\lambda_i\}_{i=1}^{N}$. Let $APL$ be the average persistence landscape. Consider the persistence measure by taking the empirical mean of the points, denoted $\mu = \frac{1}{N}\sum_i \alpha_i$. Note each $\alpha_i$ is viewed a sum of dirac delta measures. Let $\lambda$ be the continuous persistence landscapes of $\mu$. Since the average persistence landscapes are functions on $\Z_+ \times \R$, we restrict the continuous landscapes to the same domain and compare their function values. 

Recall $\Qth$ is the open quadrant $(-\infty, t-h)\times(t+h,\infty)$ and $\cQth$ is the closed quadrant $(-\infty, t-h]\times[t+h,\infty)$.

\begin{proposition}\label{apl and cpl}
    Fix $k\in \Z_+$.
    If for each $i$, $ \alpha_i(\cQth) $ contains exactly $k$ distinct points, then $ \lambda(k,t) \leq APL(k,t)$. 
\end{proposition}

\begin{proof}
    It suffice to show for $k=1$, because for every fixed $k$ 
    and each $\alpha_i$, we can consider its rank function, $\rk \alpha_i$, and replace with the following \cite{Betthauser:2022aa}\[
    \rk_k \alpha_i [a,b]: = \max\{ \rk \alpha_i[a,b]-(k-1),0 \},
    \] and its Möbius inversion is a PD whose landscape $\lambda_i^k(a,t)$ equals $\lambda_i(a+k-1,t)$ for all $a\ge 1$. The resulting mean measure of the new PDs gives rise to $\lambda^k(a,t)$ which equals $\lambda(a+k-1,t)$. 
    
    Let $t\in \R$. If $\lambda(1,t)=0$, $APL(1,t) \geq \lambda(k,t)$. Suppose $\lambda(k,t)=h>0$. 
    
    Fix $t',h'\in \R$ such that $Q_{t',h'} \subset \Qth$ and $\lambda(1,t')=h'$ is a local maxima.
    By assumption, $\overline{Q}_{t',h'} $ contains exactly 1 points from every $\alpha_i$, i.e. $\alpha_i(\cQth)\ge\alpha_i(\overline{Q}_{t',h'})=1$. Then $\lambda_i(1,t)\ge h$ by definition of landscapes. Therefore $APL(1,t) \geq h$.
    
\end{proof}

\section{Signed persistence measures and continuous persistence landscapes}

One can further generalize the notion of persistence measures to include signed measures and consider their vectorization. 

Let $\mu:\B(\Rl)\to \R$ be a $q-$tame measure satisfying that $\mu(Q_{(t,h)}\geq 0$, and the Jordan decomposition of $\mu$ consists of $q-$tame measures. The signed measure $\mu$ has well-defined continuous persistence landscapes, but they only satisfy properties (1)-(3) in Proposition \ref{pl properties}. Thus, the continuous landscapes fail to give rise to a premeasure and cannot recover the original measure $\mu$ via Carathéodory's extensions. 

However, we can consider the Jordan decomposition $\mu=\mu_+ - \mu_-$, each of which can be mapped to a continuous landscapes satisfying properties (1)-(4) in Propsition \ref{pl properties} and inverted to recover $\mu_+, \mu_-$ separately. 

\bibliography{references}
\bibliographystyle{plain}

\nocite{*}

\end{document}